\documentclass[12pt,a4paper]{amsart}
\usepackage[colorlinks,linkcolor=blue,citecolor=blue,urlcolor=red]{hyperref}
\usepackage{amssymb,latexsym,amscd,verbatim,color}
\usepackage[all]{xy}
\usepackage[dvipsnames]{xcolor}

\swapnumbers
\newtheorem{theorem}{Theorem}[subsection]

\newtheorem{lemma}[theorem]{Lemma}
\newtheorem{cor}[theorem]{Corollary}

\newcommand{\T}{\mathbb{T}} 
\newcommand{\onto}{\mbox{$\to\!\!\!\!\to$}}
\newcommand{\Tmod}{\text{$\T$-{\rm Mod}}}

\renewcommand{\marginpar}[2][]{}

\newcommand{\C}{\mathbb{C}}     
\newcommand{\Z}{\mathbb{Z}}     

\newcommand{\cA}{\mathcal{A}}
\newcommand{\cB}{\mathcal{B}}
\newcommand{\cC}{\mathcal{C}}

\newcommand{\cM}{\mathcal{M}}

\newcommand{\cR}{\mathcal{R}}

\newcommand{\cV}{\mathcal{V}}

\begin{document}

\title{Definable categories and $\T$-motives}

\author{Luca Barbieri-Viale}
\address{Dipartimento di Matematica ``F. Enriques", Universit{\`a} degli Studi di Milano\\ Via C. Saldini, 50\\ I-20133 Milano\\ Italy}
\email{luca.barbieri-viale@unimi.it}

\author{Mike Prest}
\address{School of Mathematics, University of Manchester, Oxford Road, Manchester
M13 9PL, UK}
\email{mprest@manchester.ac.uk}


\begin{abstract}
Making use of Freyd's free abelian category on a preadditive category we show that if $T:D\rightarrow \cA$ is a representation of a quiver $D$ in an abelian category $\cA$ then there is an abelian category $\cA (T)$,  a faithful exact functor $F_T: \cA (T) \to \cA$ and an induced representation $\tilde T: D \to  \cA (T)$ such that $F_T\tilde T= T$ universally. We then can show that $\T$-motives as well as Nori's motives are given by a certain category of functors on definable categories.
\end{abstract}
\thanks{The first author acknowledges the support of the {\it Ministero dell'Istruzione,  dell'Universit\`a e della Ricerca} (MIUR)  through the Research Project  (PRIN 2010-11) ``Arithmetic Algebraic Geometry and Number Theory''. 
The second author acknowledges the support of the {\it Engineering and Physical Sciences Research Council} (EPSRC) through the research grant EP/K022490/1, ``Interpretation functors and infinite-dimensional representations of finite-dimensional algebras".}

\maketitle

\section*{Introduction}

In \cite{NN}, see also \cite{HMN}, Nori, starting with the category ${\mathcal V}_k$ of algebraic varieties, i.e. separated schemes of finite type, over $k$ a subfield of the complex numbers $\C$, constructed an abelian category which is an avatar of the hypothesized category of effective homological mixed motives.  See \cite{A} for an introduction to the motivic world. 

First he builds the diagram, or {\it quiver}, which has, for vertices, triples $(X,Y,i)$ with $X\in {\mathcal V}_k$, $Y$ a closed subvariety of $X$ and $i\geq 0$ an integer.  If $f:X\rightarrow X'$ is a morphism in ${\mathcal V}_k$ with $f(Y)\subseteq Y'$ then there is a corresponding arrow from $(X,Y,i)$ to $(X',Y',i)$ for each $i$.  There is also an arrow, for each pair of closed subvarieties $Z\subseteq Y \subseteq  X$, from $(X,Y,i)$ to $(Y,Z, i-1)$ corresponding to the boundary map in the long exact sequence of singular homology groups.

He then considers the representation of this diagram, given by singular homology of a pair, mapping $(X,Y,i)$ to the finitely generated abelian group $H_i^{\rm sing}(X (\C),Y(\C))$.  From this, he constructs an abelian category {\rm EHM} through which this, and any other reasonable (co)ho-mology theory should factor.  More generally, in fact, such a universal abelian category exists for any representation in the category $R\mbox{-}{\rm mod}$, of finitely generated modules over a commutative Noetherian ring $R$.\\

\noindent {\bf Theorem {\rm (Nori)}.}  {\it Let $D$ be a quiver and $T:D\rightarrow R\mbox{-}{\rm mod}$ a representation of $D$. There is an abelian $R$-linear category ${\mathcal C}(T)$, a $R$-linear faithful exact functor $F_T: {\mathcal C} (T) \to R\mbox{-}{\rm mod}$ and a representation $\tilde T :D \to  {\mathcal C} (T)$ such that $F_T\tilde T= T$ universally.}\\

See \cite[Chapter 6]{HMN} for a detailed proof. The universal property goes as follows. Let ${\mathcal B}$ be an abelian $R$-linear category, $G:{\mathcal B}\rightarrow R\mbox{-}{\rm mod}$ a $R$-linear faithful exact functor, and $S:D\rightarrow {\mathcal B}$ a representation of $S$ in ${\mathcal B}$ such that $G S= T$.  Then there is a $R$-linear faithful exact functor $F_S: {\mathcal C}(T)\rightarrow {\mathcal B}$ unique (up to unique isomorphism) such that the following diagram commutes (up to isomorphism):
$$\xymatrix{&  {\mathcal C} (T)\ar@{.>}[d]^{F_S} \ar@/^/[dr]^{F_T} & \\
D \ar@/_1.2pc/[rr]_-{T}\ar@/^/[ur]^-{\tilde T}\ar[r]^-{S} &  {\mathcal B}\ar[r]^-{G}  & R\mbox{-}{\rm mod}\\
} $$
The original construction of ${\mathcal C}(T)$ is not straightforward and goes {\it via} a construction, involving comodules, for finite subquivers of $D$ and then taking a 2-colimit of abelian categories. Under additional conditions on $R$, e.g. if it is a field, we have that ${\mathcal C}(T)$ itself is a category of comodules over a coalgebra.  Note that in \cite{ABV} it is proven that Deligne's 1-motives can be obtained via Nori's construction.

In \cite{BVCL}, Caramello gives a proof of Nori's theorem by a very different, more general and rather direct construction, obtaining the category ${\mathcal C}(T)$ as the Barr exact completion or effectivization of the syntactic category of the regular theory of the representation $T$. This category is indeed abelian and it has the required universal property even for representations in all $R$-modules, the latter following from general properties of models of a regular theory.

Subsequently, in \cite{BV}, for any fixed base category $\cC$ along with a distinguished subcategory $\cM$, a regular homological theory $\T$ is introduced on a signature strongly related to Nori's diagram when $\cC$ is $\cV_k$ and $\cM$ is given by closed immersions. The exact completion of the regular syntactic category of $\T$ is an abelian category $\cA[\T]$ which is universal with respect to $\T$-models in abelian categories. Nori's category {\rm EHM} can be obtained as $\cA[\T_{H^{\rm sing}}]$ for the regular theory $\T_{H^{\rm sing}}$ of the model $H^{\rm sing}$ (singular homology) by adding to $\T$ all regular axioms which are valid in $H^{\rm sing}$. The category $\cA[\T]$ is the category of constructible $\T$-motives in \cite{BV}.\\

In fact, in the additive context, a direct algebraic construction may be given of both $\cC (T)$ and $\cA[\T]$, which makes use of Freyd's free abelian category on a preadditive category \cite{F}. We use \cite[Chapter 4]{PMAMS} as a reference for this, see also \cite{PCRM}. We also make use of definable additive categories which are exactly the categories of models of regular  additive theories in abelian groups. In fact, all the previously mentioned  constructions can be deduced from the following result.\\

\noindent {\bf Theorem.} {\it Let $D$ be a quiver. There is an abelian category ${\rm Ab} (D)$ and a universal representation $\Delta: D \to  {\rm Ab} (D)$, i.e. if $T:D\rightarrow \cA$ is a representation of $D$ in an abelian category $\cA$ then there is a unique exact functor $F: {\rm Ab} (D) \to \cA$ such that $F \Delta = T$. Furthermore, there is a Serre quotient $\pi :
{\rm Ab} (D)\onto \cA (T)$ along with a faithful exact functor $F_T: \cA (T) \to \cA$ and an induced representation $\tilde T: D \to  \cA (T)$ such that $F_T\tilde T= T$ universally.}\\

The universal property here can be visualized by the following commutative diagram
$$\xymatrix{&  {\rm Ab} (D)\ar@{>>}[d]^{\pi} \ar@/^/[dr]^{F} & \\
D \ar@/_1.2pc/[rr]_-{T}\ar@/^/[ur]^-{\Delta}\ar[r]^-{\tilde T}\ar@/_/[dr]_-{S} &  \cA (T)\ar[r]^-{F_T}\ar@{.>}[d]^{F_S}  & \cA\\
&  \cB\ar@/_/[ur]_-{G}& \\} $$
where $G S= T$,  $\cB$ is abelian, $G$ and $F_S$ are faithful exact. The functors $F$, $F_T$ and $F_S$ are unique up to natural equivalence.

There is an easy $R$-linear variant of this Theorem, where ${\rm Ab} (D)$ has to be replaced by a universal $R$-linear abelian category ${\rm Ab}_R (D)$, showing Nori's Theorem in the particular case when $\cA = R\mbox{-}{\rm mod}$.

The construction  of this abelian category $\cA (T)$ attached to a representation $T$ gives us, in addition, an interpretation of $\cA (T)$ as a certain category of functors on the definable category generated by $T$ (that is, the category of models of its regular theory).

We also describe another interpretation of $\cA (T)$, as the category of pp-pairs and pp-defined maps, already present as a consequence of the previous construction (we remark that ``regular formula" and ``pp formula" are alternative and equivalent terminologies). This interpretation sheds some light on the category $\cA[\T]$, presented here as a Serre quotient of ${\rm Ab}(D)$ where $D$ is the canonical diagram associated to a pair $(\cC, \cM)$ as mentioned above, and also on its Serre quotients $\cA[\T']$ obtained by adding regular axioms to the homological theory $\T$ as explained in \cite{BV}. \\

The plan of this paper is the following. In Section 1 we provide a proof of the Theorem above and its $R$-linear variant. In Section 2 we give a description of ${\rm Ab}(D)$ and $\cA (T)$ in terms of definable categories and we show the link with $\T$-motives.

\subsection*{Notation} We shall denote by ${\rm Ab}$ the category of abelian groups. For $\cA$ and $\cB$ preadditive categories $(\cA, \cB)$ shall denote the category of additive functors from $\cA$ to $\cB$ where we tacitly assume that $\cA$ is skeletally small, i.e. it  has a set of objects up to isomorphism.

\section{The universal representation of a quiver}
Here we construct the universal representation $\Delta: D \to  {\rm Ab} (D)$ of a quiver.
We show how any representation $T$ of $D$ in an abelian category $\cA$ lifts to an exact functor $F :{\rm Ab} (D)\to \cA$ between abelian categories. We also describe an $R$-linear variant  of this construction.

\subsection{Representations and modules over path algebras} Let's start with a diagram, or quiver, $D$, that is, a directed graph, given by a set $D_0$ of vertices and a set $D_1$ of arrows, plus source and target maps from $D_1$ to $D_0$.  Then a {\it representation} $T$ of $D$ in an abelian (or more generally, preadditive) category ${\mathcal A}$ is an assignment of, for each $s\in D_0$, an object $T_s \in {\mathcal A}$, and for each arrow $\alpha: s\rightarrow t$ in $D_1$, a morphism $T_\alpha :T_s \rightarrow T_t$ of ${\mathcal A}$.

The category ${\rm Rep}_{\mathcal A}(D)$ of ${\mathcal A}$-representations of $D$ has these for its objects.  A morphism $f:T\rightarrow S$ between representations is a collection $(f_s)_{s\in D_0}$ of morphisms in ${\mathcal A}$ with $f_s:T_s\rightarrow S_s$ such that, for every $\alpha\in D_1$, $\alpha:s\rightarrow t$, we have
$$\xymatrix{T_s \ar[r]^{T_\alpha} \ar[d]_{f_s} & T_t \ar[d]^{f_t} \\ S_s \ar[r]_{S_\alpha} & S_t}$$
a commutative square in the category $\cA$.

Note that any representation $T: D\to {\mathcal A}$ extends uniquely to a functor $\overline{T} : \overline{D}\to \cA$ from the {\it path category} $\overline{D}$ of $D$.

If we start with a skeletally small category ${\mathcal C}$ then by a representation of ${\mathcal C}$ in ${\mathcal A}$ we mean a {\it functor} from ${\mathcal C}$ to ${\mathcal A}$.  Note that any such functor extends uniquely to an additive functor from the abelian enrichment ${\mathbb Z}{\mathcal C}$ of ${\mathcal C}$, which is defined to have the same objects of ${\mathcal C}$ and to have, for group of morphisms from $c$ to $d$, the free ${\mathbb Z}$-module on ${\mathcal C}(c,d)$, with composition being defined in the obvious way.

If we have a skeletally small preadditive category ${\mathcal R}$, for example one of the above form ${\mathbb Z}{\mathcal C}$, then we may regard this as a ring with many objects.  In this case a representation of ${\mathcal R}$ in ${\mathcal A}$ will mean an {\it additive functor} from ${\mathcal R}$ to ${\mathcal A}$ and is usually referred to as a {\it left ${\mathcal R}$-module} $M : {\mathcal R} \rightarrow {\mathcal A}$.

If $D$ is our diagram then the abelian enrichment ${\mathbb Z}\overline{D}$ of the path category may be identified with the ${\mathbb Z}$-{\it path algebra}, ${\mathbb Z}D$, of $D$.  This is formed by taking the free ${\mathbb Z}$-module on basis $D_0 \cup D^\ast_1$ where the latter denotes the set of all {\it paths} formed from arrows in $D_1$, that is, sequences $\alpha_n \dots\alpha_1$ where the source of $\alpha_{i+1}$ is the target of $\alpha_i$ for each $i$ (so if $\alpha$ is a loop then all powers of $\alpha$ are in $D_1^\ast$).  To define the multiplication on ${\mathbb Z}D$ it is enough to define it on basis elements, where it is composition when defined, and $0$ when not (i.e.~if the target of the path $p$ is a different vertex from the source of the path $q$ then $qp=0$).  The arrows corresponding to the vertices of $D$ act as local identities. If there are just finitely vertices then ${\mathbb Z}D$ is equivalent to a ring, in general to a small preadditive category.

There is a unique extension of the representation $T :D\to \cA$ to an additive functor $M:{\mathbb Z}D \rightarrow {\mathcal A}$.  On objects, this agrees with $T$ and, if $\alpha_i\in D(d,d')$ and $n_i\in {\mathbb Z}$, then $M(\sum_i n_i\alpha_i) = \sum_i n_iT_{\alpha_i}$.

\begin{lemma}\label{rep} The category ${\rm Rep}_{\mathcal A}(D)$ is naturally equivalent to the category
$(\Z D , \cA)$ of additive functors from ${\mathbb Z}D$ to ${\mathcal A}$.  In particular, the category of representations of $D$ in ${\rm Ab}$ is naturally equivalent to the category ${\mathbb Z}D\mbox{-}{\rm Mod}$ of left ${\mathbb Z}D$-modules.
\end{lemma}
\begin{proof}
If we think of a ${\mathbb Z}D$-module as an additive functor from the preadditive category ${\mathbb Z}\overline{D}$ then, from representations of $D$ to ${\mathbb Z}D$-modules is the construction above, with the inverse being restriction of a ${\mathbb Z}D$-module to its values on the vertices and arrows of $D$.
\end{proof}

\subsection{Freyd's free abelian category}
So let us suppose that we are given a representation of one of the above forms.  For uniformity of treatment, we may as well assume this be an additive functor $M:{\mathcal R} \rightarrow {\mathcal A}$ with ${\mathcal A}$ abelian and $\cR$ preadditive.  Then there is an embedding of ${\mathcal R}$ into an abelian category (which does not depend on $M$) and an essentially unique lift to an exact functor as below.

\begin{theorem}[\protect Freyd, \cite{F}]\label{Freyd} Given a skeletally small preadditive category ${\mathcal R}$, there is a full and faithful embedding ${\mathcal R} \rightarrow {\rm Ab}({\mathcal R})$ such that, for any  additive functor $M:{\mathcal R} \rightarrow {\mathcal A}$, where ${\mathcal A}$ is abelian, there is a unique-to-natural-equivalence exact extension $F:{\rm Ab}({\mathcal R}) \rightarrow {\mathcal A}$ as follows
$$\xymatrix{{\mathcal R} \ar[r] \ar[d]_M & {\rm Ab}({\mathcal R}) \ar@{.>}[dl]^{F} \\  {\mathcal A}}$$
\end{theorem}

It may be that the functor $F$ above is not faithful.  The object kernel of $F$ is a Serre subcategory of ${\rm Ab}({\mathcal R})$, so we may consider the quotient category  ${\mathcal A}(M) =:{\rm Ab}({\mathcal R})/{\rm ker}(F)$ and we get the factorisation
$$\xymatrix{{\rm Ab}({\mathcal R}) \ar@{>>}[r]^{\pi} \ar[d]_F & {\mathcal A}(M) \ar@{.>}[dl]^{F_M}\\  {\mathcal A}}$$
 where now the new induced functor $F_M: {\mathcal A}(M) \to \cA$ is faithful and exact.

\begin{cor} Given a skeletally small preadditive category ${\mathcal R}$ and an additive functor $M:{\mathcal R} \rightarrow {\mathcal A}$, there is a small abelian category ${\mathcal A}(M)$ and a commutative diagram as shown, where $F_M$ is exact and is faithful on objects and morphisms.
\end{cor}

A universal property of ${\mathcal A}(M)$, generalising that of ${\rm Ab}({\mathcal R})$ in Theorem~\ref{Freyd}, will be given in Section \ref{secdef}.

\subsection{Proof of the Theorem}
Now apply Freyd's general framework to the preadditive category ${\mathcal R} = \Z D$ and denote the category ${\rm Ab}(\Z D)$ by ${\rm Ab}(D)$ for short.  We follow the construction and proof of Freyd's theorem as given in \cite[Chapter 4]{PMAMS}.

Note that the functor $M$ from ${\mathbb Z}D$ to  ${\mathcal A}$ induced by $T\in {\rm Rep}_{\mathcal A}(D)$  has a unique extension to an additive functor $M^+$ from the additive completion ${\mathbb Z}D^{+}$ of ${\mathbb Z}D$.  The objects of ${\mathbb Z}D^{+}$ are finite sequences of objects of $D$ and the arrows are rectangular matrices of arrows (with appropriate domains and codomains) from ${\mathbb Z}D$.  The functor $M^+$ takes $(d_1, \dots, d_n)$ to $T_{d_1} \oplus \dots \oplus T_{d_n}$ and takes a morphism $(\alpha_{ij})_{ij}$ to $(M\alpha_{ij})_{ij}$.  This may be further extended uniquely to an additive functor $M^{++}$ from the idempotent-splitting (= Karoubian = pseudoabelian) completion ${\mathbb Z}D^{++}$ to ${\mathcal A}$.  The objects of ${\mathbb Z}D^{++}$ are pairs $X:= (d,\alpha)$ with $d$ an object of ${\mathbb Z}D^+$ and $\alpha=\alpha^2$ an idempotent endomorphism of ${\mathbb Z}D^+$, the morphisms from $(d,\alpha)$ to $(e,\beta)$ being the morphisms $\gamma:d\rightarrow e$ with $\beta\gamma\alpha=\gamma$.

\begin{lemma} There is a natural equivalence between Grothendieck abelian categories
$${\rm Rep}_{\rm Ab}(D)\simeq  {\mathbb Z}D\mbox{-}{\rm Mod}=({\mathbb Z}D,{\rm Ab})\simeq ({\mathbb Z}D^+,{\rm Ab})\simeq ({\mathbb Z}D^{++},{\rm Ab})$$
\end{lemma}
\begin{proof} This follows from Lemma \ref{rep} and \cite[Chapter 2]{PMAMS}.
\end{proof}

We tacitly keep using the equivalence ${\mathbb Z}D\mbox{-}{\rm Mod}\simeq {\mathbb Z}D^{++}\mbox{-}{\rm Mod}:= ({\mathbb Z}D^{++},{\rm Ab})$ in what follows. Let ${\mathbb Z}D\mbox{-}{\rm mod}$ denote the category $({\mathbb Z}D,{\rm Ab})^{\rm fp}$ of {\it finitely presented} left ${\mathbb Z}D$-modules.  Since the Yoneda embedding $X \mapsto (X,-)$ from ${\mathbb Z}D^{++}$ to $\big({\mathbb Z}D^{++}\mbox{-}{\rm Mod} \big)^{\rm op}$ is an anti-equivalence of ${\mathbb Z}D^{++}$ with the category of finitely generated projectives in ${\mathbb Z}D\mbox{-}{\rm Mod}$, and since these are generating, the finitely presented modules are those with a projective presentation as the cokernel of $(Y,-) \xrightarrow{(\gamma,-)} (X,-)$ for some $\gamma:X\rightarrow Y$ in ${\mathbb Z}D^{++}$.

 We extend $M$, equivalently $M^{++}$, to a left exact functor
 $$F':\big({\mathbb Z}D \mbox{-}{\rm mod} \big)^{\rm op} \rightarrow {\mathcal A}$$ that is, to a right exact functor from ${\mathbb Z}D \mbox{-}{\rm mod}$ to $\cA$, as follows (cf.~the proof of Theorem 4.3 in \cite{PMAMS} and the commentary before that).
Let $N\in {\mathbb Z}D \mbox{-}{\rm mod}$.  Take $\gamma :X\rightarrow Y$ in ${\mathbb Z}D^{++}$ inducing, under the Yoneda embedding, a projective presentation $$(Y,-) \xrightarrow{(\gamma,-)} (X,-) \rightarrow N \rightarrow 0$$ of $N$.  Then define 
$$F'(N):={\rm ker}(M^{++}X\xrightarrow{M^{++}\gamma}M^{++}Y)$$  Using projectivity of representable functors, there is an induced action on morphisms and one checks that the action of $F'$ on objects is independent of the chosen projective presentation and that the action on morphisms is well-defined.

We apply again a Yoneda functor $N^{\rm o}\rightarrow (N,-)$ from $\big({\mathbb Z}D \mbox{-}{\rm mod} \big)^{\rm op} $ to $( {\mathbb Z}D \mbox{-}{\rm mod},{\rm Ab})^{\rm fp}$, where we use superscript $^{\rm o}$ to indicate objects and morphisms in the opposite category. Finally $F'$ is extended to a right exact and, one may check (directly or using the equivalence, in Section \ref{secdef}, between $( {\mathbb Z}D \mbox{-}{\rm mod},{\rm Ab})^{\rm fp}$ and the category of pp-pairs), exact, in addition to well-defined, functor $$F: ( {\mathbb Z}D \mbox{-}{\rm mod},{\rm Ab})^{\rm fp}\to {\mathcal A}$$ by sending $\Theta\in ({\mathbb Z}D \mbox{-}{\rm mod},{\rm Ab})^{\rm fp}$ to $$F(\Theta):= {\rm coker}(F'N^{\rm o}\xrightarrow{F'g^{\rm o}} F'P^{\rm o})$$ where $P\xrightarrow{g}N$ in ${\mathbb Z}D^{++}\mbox{-}{\rm mod}$ is such that $(N,-)\xrightarrow{(g,-)} (P,-) \rightarrow \Theta  \rightarrow 0$ is a projective presentation in $( {\mathbb Z}D \mbox{-}{\rm mod},{\rm Ab})^{\rm fp}$.\\

In summary, we obtain the following commutative diagram
$$\xymatrix{D \ar@/_/[rrd]_T \ar[r] \ar@/^2.5pc/[rrr]^{\Delta}&\Z D \ar[r] \ar[rd]_M & ({\mathbb Z}D \mbox{-}{\rm mod})^{\rm op} \ar[r] \ar[d]^{F'} & ({\mathbb Z}D \mbox{-}{\rm mod},{\rm Ab})^{\rm fp}:= {\rm Ab}(D)\ar@/^/[dl]^{F} \\ & & {\mathcal A}}$$
where the functor from $ {\mathbb Z}D$ to ${\rm Ab}(D)$ factors through $({\mathbb Z}D^{++}\mbox{-}{\rm mod})^{\rm op}$ and/or equivalently $({\mathbb Z}D\mbox{-}{\rm mod})^{\rm op} $ and it is the composition of two Yoneda embeddings. Thus, given $M:{\mathbb Z}D \longrightarrow {\mathcal A} $ induced by $T\in {\rm Rep}_{\mathcal A}(D)$ with $ {\mathcal A} $ abelian, one obtains the functor $F$ which is exact, along with its restriction $F'$ which is only left exact.

Now let $\ker (F) :=\{ \Theta\in {\rm Ab}(D): F(\Theta)=0\}$.  This is a Serre subcategory of ${\rm Ab}(D)$ and so, by the universal property of the quotient $$\pi: {\rm Ab}(D)\rightarrow {\mathcal A}(T):= {\mathcal A}(M) = {\rm Ab}(D)/\ker (F)$$ there is an essentially unique, exact, functor $F_T:{\mathcal A}(T) \rightarrow {\mathcal A}$ with $F_T\pi =F$ as claimed. Clearly, the induced representation $\tilde T =: \pi \Delta$ from $D$ to  $\cA (T)$ is such that
$F_T\tilde T= F_T\pi \Delta = F \Delta = T$.

Furthermore, if we have $S\in {\rm Rep}_{\mathcal B}(D)$ with $\cB$ abelian and $G:\cB\to \cA$ faithful exact such that $GS =T$ then let $H : {\rm Ab}(D)\to \cB$ be the exact functor induced by Freyd's theorem such that $H\Delta = S$. So $G H\Delta = T$ and by uniqueness of $F$ we have a natural equivalence $G H \cong F$. This implies that ${\mathcal A}(T) = {\mathcal A}(S)$ since $\ker (F) = \ker (H)$, $G$ being faithful exact. We then have $\tilde T= \tilde S$ and get $F_S : {\mathcal A}(T)\to \cB$ faithful exact such that $G F_S$ is naturally equivalent to $F_T$. The Theorem is then clear.

\subsection{Nori's category and the $R$-linear case}\label{secRlin}
Note that for $R$ a commutative unitary ring we now can get an $R$-linear structure $RD$ in the same way as we did for $\Z D$ by considering the \emph{$R$-path algebra}. This is the small preadditive $R$-linear category $RD$ given by the path category $\overline{D}$. We here define $RD$ to have the same objects as $\overline{D}$ and to have, for $R$-module of morphisms from $c$ to $d$, the free $R$-module on $\overline{D}(c,d)$, with composition being defined as usual.

Given this, and an additive $R$-linear category ${\mathcal A}$, the category ${\rm Rep}_{\mathcal A}(D)$ is then naturally equivalent to the category $(RD , \cA)$ of additive $R$-linear functors from $RD$ to $\cA$. Now we have natural equivalences $$(RD,R\mbox{-}{\rm Mod})\simeq (RD^+, R\mbox{-}{\rm Mod})\simeq (RD^{++},R\mbox{-}{\rm Mod})$$
as above. Consider the category $(RD,R\mbox{-}{\rm Mod})^{\rm fp}$ of finitely presented $RD$-modules and set
$${\rm Ab}_R(D):= ((RD,R\mbox{-}{\rm Mod})^{\rm fp}, R\mbox{-}{\rm Mod})^{\rm fp}$$
Thus given $M:RD \longrightarrow {\mathcal A} $ induced by $T\in {\rm Rep}_{\mathcal A}(D)$ with $ {\mathcal A}$ abelian and $R$-linear, one obtains, as above, the functor $F: {\rm Ab}_R(D) \to \cA$ which is exact and $R$-linear. We then set
$$\cA (T):= {\rm Ab}_R(D)/\ker F$$ so that $F_T: \cA (T)\to \cA$ is also faithful and  $R$-linear. As a particular case, we get Nori's theorem as a corollary of our Theorem:

\begin{cor} For a Noetherian ring $R$ and $\cA = R\mbox{-}{\rm mod}$ we get $\cC (T) \simeq \cA (T)$.
\end{cor}

Let us also note that when the preadditive category ${\mathcal R}$ has an $R$-linear structure then so does ${\rm Ab}({\mathcal R})$ and hence ${\rm Ab}_R({\mathcal R})$ is naturally equivalent (as an $R$-linear category) to ${\rm Ab}({\mathcal R})$ equipped with that $R$-linear structure.  To see the $R$-linear structure on ${\rm Ab}({\mathcal R})$ directly we can follow its construction.  A finitely presented ${\mathcal R}$-module is the cokernel of a morphism $(Y,-) \xrightarrow{(\gamma,-)} (X,-)$ with $\gamma:X \rightarrow Y$ in ${\mathcal R}$.  The $R$-linear structure of ${\mathcal R}$ means that $R$ acts as endomorphisms of the identity functor of ${\mathcal R}$, hence any additive functor from ${\mathcal R}$ to ${\rm Ab}$ factors through $R\mbox{-}{\rm Mod}$.  So there are induced $R$-module structures on $(Y,-)$ and $(X,-)$, and $(\gamma,-)$ will be a morphism of $R$-modules, hence the cokernel also carries an $R$-module structure.  Similarly for the second stage of the construction.  We deduce that every object of ${\rm Ab}({\mathcal R})$ is an $R$-module and all morphisms of ${\rm Ab}({\mathcal R})$ are $R$-linear.

\section{Theoretical motives via definable categories}

In the additive context the link between regular theories and their categories of models can be expressed in terms of Serre subcategories and corresponding localisations of the free abelian category.  This uses the equivalent views of these abelian categories as, on the one hand categories of pp-sorts and, on the other, as localisations of the category of functors on finitely presented models.

\subsection{Definable additive categories}\label{secdef}

If ${\mathcal R}$ is a skeletally small preadditive category then there is an equivalence (see e.g.~\cite[10.2.30]{PP}) between ${\rm Ab}({\mathcal R})$ as constructed above and the category of pp-pairs, equivalently regular sequents, for the theory of ${\mathcal R}$-modules.  As a sequent, such a pair can be written $\vdash_{\overline{x}} \phi \rightarrow \psi$ where, without of loss of generality, it can be assumed that $\vdash_{\overline{x}} \psi \rightarrow \phi$ already is valid in every ${\mathcal R}$-module.  Here $\phi$ and $\psi$ are pp (for ``positive primitive") formulas in the same free variables, also referred to as regular formulas.  As a pp-pair, this is written $\phi/\psi$ (or $\phi(\overline{x})/\psi(\overline{x})$ in order to show the free variables).

This pp-pair notation already refers to the fact that every pp-pair $\phi/\psi$, being an object of ${\rm Ab}({\mathcal R})$, can be regarded as a functor, $F_{\phi/\psi}$, from ${\mathcal R}\mbox{-}{\rm mod}$ to ${\rm Ab}$, indeed from ${\mathcal R}\mbox{-}{\rm Mod}$ to ${\rm Ab}$.  Namely it is that which takes $M \in {\mathcal R}\mbox{-}{\rm Mod}$ to $\phi(M)/\psi(M)$, where $\phi(M)$ denotes the solution set of $\phi$ in $M$.  When ${\mathcal R}$ is a ring with one object $\phi(M)$ is a subgroup of $M^n$, $n$ being the number of free variables in $\phi$; in general it is a subgroup of the product of the sorts of $M$ corresponding to the variables $\overline{x}$.  Solution sets to pp formulas are preserved by morphisms, so there is an induced action giving the value of $F_{\phi/\psi}$ on morphisms.  The phrase $\phi/\psi$ is {\it closed} on $M$ is used to mean that the sequent $\vdash_{\overline{x}} \phi \rightarrow \psi$ is valid in $M$, since its being valid is equivalent to the group $\phi(M)/\psi(M)$ being trivial.

These pp-pairs are the objects of the category of pp-pairs, denoted $_{\mathcal R}{\mathbb L}^{\rm eq+}$ and equivalent to ${\rm Ab}({\mathcal R})$; the arrows are the pp-defined maps between pp-pairs - that is, the pp-defined relations between such pairs which are functional.  Almost by definition, this category is naturally equivalent to the effectivisation of the regular syntactic category for the theory of ${\mathcal R}$-modules.  Note that the formal language, and hence the meaning of pp=regular formula and the syntactic category, does depend on the choice of language.  For instance the language based on ${\mathbb Z}D$ will have fewer sorts than that based on ${\mathbb Z}D^{++}$ but they are equivalent in that any (pp) formula of the larger language is equivalent in every ${\mathbb Z}D$-module to a (pp) formula in the smaller language, at least if one allows the addition of new sorts.  In model theory this addition of definable new sorts is known as the imaginaries, or $^{\rm eq}$ construction (the notation $^{\rm eq+}$ indicates the construction restricted to pp formulas).  New sorts are formed by introducing finite products of existing sorts and factoring by definable equivalence relations -- the process referred to as effectivisation in the more category-theoretic model theory literature.  The various initial choices for language/syntactic category do all lead to the same abelian category of pp-sorts which may, therefore, be regarded as the category underlying the richest, and canonical, language for ${\mathcal R}$-modules.

Given an ${\mathcal R}$-module $M$, the quotient category ${\mathcal A}(M)$ similarly may be regarded as the category of sorts and function symbols for the canonical language for $M$ (and for the modules in the definable category that $M$ generates, see below).  This gives us an interpretation of the exact functor which we have denoted $F$ (or $F_M$) as the unique extension, $M^{\rm eq+}$, of $M$ to a structure for the canonical language:  it assigns to each sort $\phi/\psi$ its value, $\phi(M)/\psi(M)$, on $M$ and to each arrow from $\phi(\overline{x})/\psi(\overline{x})$ to $\phi'(\overline{y})/\psi'(\overline{y})$ the corresponding additive function from $\phi(M)/\psi(M)$ to $\phi'(M)/\psi'(M)$ (where this function is definable by some pp formula $\theta(\overline{x},\overline{y})$).  For more discussion of these languages see \cite{PCRM}.

Next recall that if ${\mathcal S}$ is a Serre subcategory of an abelian category ${\mathcal A}$ then the standard construction of the localised category ${\mathcal A}/{\mathcal S}$ changes the morphisms but not the objects.  Therefore, in the case that $M$ is an ${\mathcal R}$-module and the Serre subcategory is the kernel, ${\mathcal S}_M$ of $\pi:{\rm Ab}({\mathcal R}) \rightarrow {\mathcal A}(M)$, the objects of the localised category ${\mathcal A}(M)$ are the pp-pairs as before.  The morphisms are the pp-defined relations between sorts which are provably functional in the regular theory of $M$, equivalently which are functional when evaluated on $M$, equivalently which are functional on every object in the definable category generated by $M$ (we define this next).  Note that the pp-pairs which are closed on $M$ are exactly those which are objects in ${\mathcal S}_M$.  Furthermore, a set of pairs/regular sequents which generates ${\mathcal S}_M$ as a Serre subcategory is exactly a set of regular axioms which generates the regular theory, $\T_M$, of $M$.

For $M$ as above, the {\it definable category} $\langle M \rangle$ {\it generated by} $M$ is the full subcategory of ${\mathcal R}\mbox{-}{\rm Mod}$ on those objects which are models of the regular theory of the ${\mathcal R}$-module $M$.  These are the modules which satisfy all the sequents $\vdash \phi \rightarrow \psi$ in $\T_M$, equivalently those modules on which each pp-pair $\phi/\psi$ in ${\mathcal S}_M$ is closed.  This category has an algebraic characterisation, see \cite[3.4.7]{PP}, as the smallest full subcategory of ${\mathcal R}\mbox{-}{\rm Mod}$ containing $M$ which is closed under direct products, direct limits and pure submodules.  It is also equivalent to the category of exact functors from ${\mathcal A}(M)$ to ${\rm Ab}$, i.e.
$$\langle M \rangle\simeq {\rm Ex}({\mathcal A}(M), {\rm Ab})$$
(this follows from the universal property of the abelian categories involved).

As described above, the equivalence of the category ${\mathcal A}(M)$ with the category of pp-pairs and pp-definable maps for the regular theory of $M$ allows us to view ${\mathcal A}(M)$ as a category of functors from $\langle M \rangle$ to ${\rm Ab}$.  In fact, \cite[12.10]{PMAMS}, it is the category of those additive functors from $ \langle M \rangle$ to ${\rm Ab}$ which commute with direct products and direct limits.

Each object $M'$ in $ \langle M \rangle$ has regular theory containing that of $M$; these will be equal iff $ \langle M' \rangle = \langle M \rangle$, that is, in the regular theories terminology, if and only if $M'$ is a conservative model of the regular theory of $M$.  Otherwise $M'$ satisfies more regular sequents than $M$, so ${\mathcal S}_{M'} \supsetneq {\mathcal S}_M$.  In that case, $ \langle M' \rangle$ will be a proper definable subcategory of $ \langle M \rangle$ and, on the functor category side, ${\mathcal A}(M')$ will be a proper Serre quotient of ${\mathcal A}(M)$.  A Serre-generating subset of the kernel, ${\mathcal S}_{M'}/{\mathcal S}_M$, of that quotient map is a set of regular axioms which must be added to the regular theory of $M$ in order to logically generate the regular theory of $M'$.  We remark that every proper definable subcategory of $ \langle M \rangle$ has the form $ \langle M' \rangle$ for some $M' \in  \langle M \rangle$ (\cite[3.4.12]{PP}).  We may express this relation between $\langle M \rangle$ and ${\mathcal A}(M)$ by the following universal property.

\begin{theorem}\label{defcatuniv} Suppose that ${\mathcal R}$ is an $R$-linear preadditive category, where $R$ is a commutative unital ring.  Let $M:{\mathcal R} \rightarrow R\mbox{-}{\rm Mod}$ be an $R$-linear representation.  Let $M'\in \langle M \rangle$, that is, suppose that $M'$ is a model of the regular theory of $M$.  Then there is a unique, to natural equivalence, exact $R$-linear functor $G_{M'}:{\mathcal A}(M) \rightarrow R\mbox{-}{\rm Mod}$ such that the following diagram 
$$\xymatrix{{\mathcal R} \ar[r] \ar[d]_{M'} & {\mathcal A}(M) \ar@{.>}[dl]^{G_{M'}} \\  R\mbox{-}{\rm Mod}}$$
commutes. The functor $G_{M'}$ will be faithful precisely if $\langle M' \rangle = \langle M \rangle$, that is, if and only if $M$ and $M'$ have the same regular theory.
\end{theorem}
\begin{proof} Recall from Section \ref{secRlin} that ${\rm Ab}({\mathcal R})$ and hence $\cA(M)$ carries an $R$-linear structure.

For existence, we take the composition of the canonical localisation $\pi_0:{\mathcal A}(M) \rightarrow {\mathcal A}(M')$ (which exists since, by hypothesis, ${\mathcal S}_M \subseteq {\mathcal S}_{M'}$) with the functor $F_{M'}$ from the diagram after \ref{Freyd}.

Suppose that the exact functor $G:{\mathcal A}(M) \rightarrow R\mbox{-}{\rm Mod}$ also makes the diagram commute.  Precomposing with the localisation $\pi:{\rm Ab}({\mathcal R}) \rightarrow {\mathcal A}(M)$, we deduce from \ref{Freyd} that $G\pi$ and $G_{M'}\pi$ are naturally equivalent and hence, by the universal property of the localisation $\pi$, that $G$ is naturally equivalent to $G_{M'}$.

Since $G_{M'}$ factors through the localisation ${\mathcal A}(M) \rightarrow {\mathcal A}(M')$ the last statement is direct from the discussion above, describing the kernel of this localisation in terms of regular theories.
\end{proof}

It can be useful to have a statement like \ref{defcatuniv} but starting with an $R$-linear representation $M:{\mathcal R} \rightarrow {\mathcal G}$ where ${\mathcal G}$ is any Grothendieck abelian $R$-linear category.  In fact, the proof above works just as well in that case.  Here is the diagram (which also illustrates the proof of \ref{defcatuniv}).
$$\xymatrix{& {\mathcal R} \ar[r]^\Delta \ar[dd]^<<<<<{M'} \ar[dl]_M & {\rm Ab}({\mathcal R}) \ar[dr]^\pi \ar[dd]^<<<<<{\pi'}\\
{\mathcal G} \ar@{=}[dr] & & & \cA (M) \ar[lll]_<<<<{F_M} \ar@{.>}[dl]^{\pi_0} \ar@{.>}[dll]_>>>>>>>>{G} \\
& {\mathcal G} &  \cA (M') \ar[l]^{F_{M'}}}$$

\subsection{$\T$-motives}
Consider now a category $\cC$ and a distinguished subcategory $\cM$ of $\cC$. Let $\cC^{\square}$ be the \emph{category of pairs} with objects the arrows in $\cM$ and morphisms the commutative squares of $\cC$. We shall denote $(X,Y)$ an object of $\cC^{\square}$, i.e.  a morphism $f: Y \to X$ of $\cM$, and $\square: (X,Y) \to (X', Y')$ a commutative square.

Following Nori, we can define a quiver $D$, which we can call the {\it Nori diagram of}\, $\cC^{\square}$, with vertices the triples $(X, Y, i)$ for each $i\in \Z$ and arrows $(X,Y, i) \to (X', Y', i)$ associated to arrows $\square$ in $\cC^{\square}$ with additional arrows $(X,Y,i)\to (Y,Z, i-1)$ corresponding to $\partial : (Y,Z) \to (X,Y)$ the morphism of $\cC^{\square}$ given by $f : Z \to Y$ and $g: Y \to X$ objects of $\cC^{\square}$.

We let $\T$ be the regular homological theory defined in \cite[\S 2]{BV}. In fact, for such a category $\cC^{\square}$ and/or its Nori diagram $D$ there is a corresponding, slightly richer, signature $\Sigma$ where vertices are sorts and arrows are function symbols. We let $\cA[\T]$ be the abelian category of constructible $\T$-motives as defined in \cite[\S 4]{BV}. This is the effectivization of the regular syntactic category.

 A model $H$ of the regular theory $\T$ in an abelian category $\cA$ yields a representation $H\in {\rm Rep}_{\mathcal A}(D)$. In particular, for $\cA [\T]$ and the universal model $H^{\T}$ (see \cite[4.1.5]{BV}) we get a representation $H^{\T}\in {\rm Rep}_{\cA [\T]}(D)$. We can apply our Theorem to $D$, the Nori diagram of $\cC^{\square}$, and we thus obtain an exact functor $F:  {\rm Ab} (D)\to \cA [\T]$ lifting $H^{\T}$, the abelian category $$\cA (\T):= \cA(H^\T)$$ defined to be the Serre quotient of ${\rm Ab} (D)$ by $\ker (F)$ and the induced  faithful exact functor $F_{\T} : \cA (\T)\to \cA [\T]$ filling in the following commutative diagram

$$\xymatrix{&  {\rm Ab} (D)\ar@{>>}[d]^{\pi} \ar@/^/[dr]^{F} & \\
D \ar@/_1.2pc/[rr]_-{H^{\T}}\ar@/^/[ur]^-{\Delta}\ar[r]^-{\tilde H^{\T}} &  \cA (\T)\ar@{.>}[r]^-{F_{\T}}  & \cA [\T]\\
} $$
\begin{theorem}\label{compare}
The functor $F_{\T} : \cA (\T)\xrightarrow{\simeq} \cA [\T]$ is an equivalence.
\end{theorem}

\begin{proof} Since $F_{\T}$ is a faithful exact functor it preserves and reflects validity of regular sequents. Since $F_{\T}\tilde H^{\T}= H^{\T}$ we then obtain that $\tilde H^{\T}:= \pi\Delta \in {\rm Rep}_{\cA (\T)}(D)$ is actually a model of $\T$ in the abelian category
$\cA (\T)$. Now recall (see \cite[4.1.3]{BV}) that there is a natural equivalence, where $\Tmod (\cA (\T))$ denotes the category of $\T$-models in $\cA (\T)$,
$$\Tmod (\cA (\T)) \cong {\rm Ex} (\cA [\T], \cA (\T))$$
so that we also obtain an exact functor $G : \cA [\T]\to \cA (\T)$ corresponding to the $\T$-model $\tilde H^{\T}$ above. By naturality, the composition $F_{\T}G$ is the identity since $F_{\T}\tilde H^{\T}= H^{\T}$ and the universal model corresponds to the identity. On the other hand, by construction of the equivalence between $\T$-models and exact functors, we have that $GH^{\T}\cong \tilde H^{\T}$. Therefore, by uniqueness, using the universal property of the Theorem, applied to $\tilde H^{\T}\in {\rm Rep}_{\cA (\T)}(D)$, we then get that
$GF_{\T}$ is naturally isomorphic to the identity. \end{proof}

A similar argument works for the theories $\T'$ obtained from $\T$ by adding regular axioms on the same signature. Since the universal model $H^{\T'}$ of $\T'$ in $\cA [\T']$ is also a $\T$-model we get a functor $\cA [\T]\to \cA [\T']$ sending $H^{\T}$ to $H^{\T'}$. We thus easily obtain from Theorem \ref{compare} the following:
\begin{cor} If $\T'$ is obtained from $\T$ by adding regular axioms on the same signature then $\cA [\T']$ is a Serre quotient of $\cA [\T]$. Furthermore, $\T'$ is a conservative extension of $\T$ if and only if $\cA [\T]\cong \cA [\T']$ is an equivalence.
\end{cor}

Note that this is also the case for the regular theory $\T_H$ of a model $H\in \Tmod (\cA)$ obtained by adding all regular axioms which are valid in the model $H$. As we regard a model as a representation we also get a corresponding commutative diagram

$$\xymatrix{&  {\rm Ab} (D)\ar@{>>}[d]^{\pi} \ar@/^/[dr]^{}\ar@/^/[rd]^{}\ar@/^1.2pc/[rrd]^{}\ar@/^1.5pc/[rrrd]^{F}  & \\
D \ar@/_1.8pc/[rrrr]_-{H}\ar@/^/[ur]^-{\Delta}\ar[r]^-{\tilde H} &  \cA (H)\ar@{.>}[r]^-{} & \cA (\T_H)\ar@{.>}[r]^-{}  & \cA [\T_H]\ar@{.>}[r]^-{r_H} &\cA\\
} $$
where the functor $F$ is induced by  $H\in {\rm Rep}_{\mathcal A}(D)$ and the exact functor $r_H: \cA [\T_H]\to \cA$ corresponding to $H\in \Tmod (\cA)$ is faithful as, clearly, $H$ is a conservative $\T_H$-model.  Arguing as above, $\cA (\T_H) \rightarrow \cA [\T_H]$ also is an equivalence. Since $r_H$ is faithful, $\cA(H)$ and $\cA [\T_H]$ are quotients of ${\rm Ab} (D)$ by the same Serre subcategory, hence are equivalent, and,  so finally:
\begin{cor}
The functors $\cA(H)\xrightarrow{\simeq} \cA (\T_H)\xrightarrow{\simeq} \cA [\T_H]$ are equivalences and the composition with $r_H$ coincides with the canonical faithful exact functor $F_H$ attached to $H\in {\rm Rep}_{\mathcal A}(D)$.
\end{cor}

In particular, applying this to the case of Nori's singular homology representation $H^{\rm sing}$ in $\cA = {\rm Ab}$ we get 
\begin{cor}${\rm EHM}\cong \cA (H^{\rm sing})\cong \cA (\T_{H^{\rm sing}})\cong \cA [\T_{H^{\rm sing}}]$
\end{cor}


\end{document}